\documentclass[reqno,a4paper]{amsart}
\usepackage{amssymb}
\usepackage{amsmath}
\newcommand{\angles}[1]{\langle #1 \rangle}
\newcommand{\half}{\frac{1}{2}}

\newcommand{\R}{\mathbb{R}}

\begin{document} 
\newtheorem{prop}{Proposition}[section]
\newtheorem{Def}{Definition}[section] \newtheorem{theorem}{Theorem}[section]
\newtheorem{lemma}{Lemma}[section] \newtheorem{Cor}{Corollary}[section]

\title[Maxwell-Chern-Simons-Higgs in Lorenz gauge]{\bf Local well-posedness for the Maxwell-Chern-Simons-Higgs system in Fourier-Lebesgue spaces}
\author[Hartmut Pecher]{
{\bf Hartmut Pecher}\\
Fachbereich Mathematik und Naturwissenschaften\\
Bergische Universit\"at Wuppertal\\
Gau{\ss}str.  20\\
42119 Wuppertal\\
Germany\\
e-mail {\tt pecher@uni-wuppertal.de}}
\date{}

\begin{abstract}
We consider local well-posedness for the Maxwell-Chern-Simons-Higgs system in Lorenz gauge for data with minimal regularity assumptions in Fourier-Lebesgue spaces $\widehat{H}^{s,r}$ , where $\|u\|_{\widehat{H}^{s,r}} := \| \langle \xi \rangle^s \widehat{u}(\xi)\|_{L^{r'}}$ , and $r$ and $r'$ are dual exponents.  We show that the gap between this regularity and the regularity with respect to scaling shrinks in the case $r>1$ , $r \to 1$ compared to the classical case $r=2$ .
\end{abstract}
\maketitle
\renewcommand{\thefootnote}{\fnsymbol{footnote}}
\footnotetext{\hspace{-1.5em}{\it 2020 Mathematics Subject Classification:} 
35Q40, 35L70 }
\normalsize 
\setcounter{section}{0}
\section{Introduction and main results}
\noindent The Lagrangian of the (2+1)-dimensional Maxwell-Chern-Simons-Higgs model which was proposed in \cite{LLM} is given by
\begin{align*}
{\mathcal L}  = &-\frac{1}{4} F^{\mu \nu} F_{\mu \nu} + \frac{\kappa}{4} \epsilon^{\mu \nu \rho} F_{\mu \nu} A_{\rho} + D_{\mu}\phi\overline{D^{\mu} \phi} \\
& + \frac{1}{2} \partial_{\mu}N \partial^{\mu} N - \frac{1}{2}(e|\phi|^2+\kappa N - ev^2)^2 - e^2 N^2 |\phi|^2 
\end{align*}
in Minkowski space ${\mathbb R}^{1+2} = {\mathbb R}_t \times {\mathbb R}_x^2$ with metric $g_{\mu \nu} = diag(1,-1,-1)$. We use the convention that repeated upper and lower indices are summed, Greek indices run over 0,1,2 and Latin indices over 1,2. Here 
\begin{align*}
D_{\mu}  & := \partial_{\mu} - ieA_{\mu} \\
 F_{\mu \nu} & := \partial_{\mu} A_{\nu} - \partial_{\nu} A_{\mu} 
\end{align*}
Here $F_{\mu \nu} : {\mathbb R}^{1+2} \to {\mathbb R}$ denotes the curvature, $\phi : {\mathbb R}^{1+2} \to {\mathbb C}$ and $N: {\mathbb R}^{1+2} \to {\mathbb R}$ are scalar fields, and $A_{\mu} : {\mathbb R}^{1+2} \to {\mathbb R}$ are the gauge potentials. $e$ is the charge of the electron and $\kappa > 0$ the Chern-Simons constant, $v$ is a real constant. We use the notation $\partial_{\mu} = \frac{\partial}{\partial x_{\mu}}$, where we write $(x^0,x^1,...,x^n) = (t,x^1,...,x^n)$ and also $\partial_0 = \partial_t$. $\epsilon^{\mu \nu \rho}$ is the totally skew-symmetric tensor with $\epsilon^{012} = 1$.

The corresponding Euler-Lagrange equations are given by
\begin{align}
\label{2.1}
& \partial_{\lambda} F^{\lambda \rho} + \frac{\kappa}{2} \epsilon^{\mu \nu \rho} F_{\mu \nu} + 2e Im( \phi  \overline{D^{\rho} \phi})  = 0\\
\label{2.2}
&D_{\mu} D^{\mu} \phi  + U_{\overline{\phi}}(|\phi|^2,N) = 0 \\
\label{2.3}
& \partial_{\mu} \partial^{\mu} N + U_N (|\phi|^2,N) = 0 \, ,
\end{align}
where
\begin{align*}
U_{\overline{\phi}}(|\phi|^2,N) &= (e|\phi|^2+\kappa N -ev^2)\phi + e^2 N^2 \phi \\
U_N(|\phi|^2,N) & = \kappa(e|\phi|^2 + \kappa N -ev^2) + 2e^2N |\phi|^2 \, .
\end{align*}
(\ref{2.1}) can be written as follows
\begin{align}
\label{1.1*}
&-\Delta A_0 + \partial_t(\partial_1 A_1 +\partial_2 A_2) + \kappa F_{12} + 2e Im(\phi \overline{D^0 \phi})  = 0 \\
\label{1.2*}
&(\partial_t^2-\partial_2^2)A_1 - \partial_1(\partial_t A_0 -\partial_2 A_2) - \kappa F_{02} + 2e Im(\phi \overline{D^1 \phi}) = 0 \\
\label{1.3*}
&(\partial_t^2-\partial_1^2)A_2 - \partial_2(\partial_t A_0 -\partial_1 A_1) + \kappa F_{01} + 2e Im(\phi \overline{D^2 \phi}) = 0 \, .
\end{align}
The initial conditions are
\begin{align}
\label{IC}
&A_{\nu}(0) = a_{\nu 0} \quad , \quad (\partial_t A_{\nu})(0) = a_{\nu 1}  \quad , \quad \phi(0) = \phi_0 \quad , \quad (\partial_t \phi)(0) = \phi_1 \\ 
\nonumber
&N(0) = N_0 \quad , \quad (\partial_t N)(0) = N_1   \, .
\end{align}
The Gauss law constraint (\ref{1.1*}) requires the initial data to fulfill the following condition:
\begin{equation}
\Delta a_{00} - \partial_1 a_{11} - \partial_2 a_{21} -\kappa(\partial_1 a_{20} - \partial_2 a_{10}) - 2e Im(\phi_0 \overline{\phi}_1) +2e^2 a_{00}|\phi_0|^2 = 0 \,.
\end{equation} 
The energy $E(t)$ of (\ref{2.1}),(\ref{2.2}),(\ref{2.3}) is (formally) conserved, where
\begin{align*}
E(t) = &\int \frac{1}{2} \sum_i F_{0i}^2(x,t) + \frac{1}{2} F_{12}^2(x,t) \\&+  \sum_{\mu} |D_{\mu} \phi(x,t)|^2 + \sum_{\mu} |\partial_{\mu} N(x,t)|^2 + U(|\phi|^2,N)(x,t) dx
\end{align*}
with
$$ U(|\phi|^2,N) = \frac{1}{2}(e|\phi|^2+\kappa N -ev^2)^2 +e^2 N^2 |\phi|^2 \, . $$
There are two possible natural asymptotic conditions to make the energy finite: either the "nontopological" boundary condition
$ (\phi,N,A) \rightarrow (0,\frac{ev^2}{\kappa},0)$ as $|x| \to \infty $
or the "topological" boundary condition
$(|\phi|^2,N,A) \rightarrow (v^2,0,0)$ as $|x| \to \infty \, . $

We decide to study the "nontopological" boundary condition. Replacing $N$ by $N-\frac{ev^2}{\kappa}$ and denoting it again by $N$  we obtain $(\phi,N,A) \to (0,0,0)$ as $|x| \to \infty$, thus leading to solutions in standard Sobolev spaces, and in (\ref{2.2}),(\ref{2.3}) we now have
\begin{align}
\label{*}
U_{\overline{\phi}}(|\phi|^2,N) &= (e|\phi|^2+\kappa N)\phi + e^2(N+\frac{ev^2}{\kappa})^2 \phi \\
\label{****}
 U_N(|\phi|^2,N) &= \kappa (e|\phi|^2 + \kappa N) +2e^2(N+\frac{ev^2}{\kappa}) |\phi|^2 
\end{align}

For the "topological" boundary condition the problem can also be reduced to a system for $(\phi,N,A)$ which fulfills $(\phi,N,A) \to (0,0,0)$ as $|x| \to \infty$ for a modified function $\phi$, if one makes the assumption that $\phi \to \lambda \in{\mathbb C}$ as $|x| \to \infty$ with $|\lambda| = v$. In this case one simply replaces $\phi$ by $\phi - \lambda$. For details we refer to Yuan's paper \cite{Y}. 

The equations (\ref{2.1}),(\ref{2.2}),(\ref{2.3}) are invariant under the gauge transformations
$$ A_{\mu} \rightarrow A'_{\mu} = A_{\mu} + \partial_{\mu} \chi \, , \, \phi \rightarrow \phi' = \exp(ie\chi) \phi \, , \, D_{\mu} \rightarrow D'_{\mu} = \partial_{\mu}-ieA'_{\mu} \, . $$
We consider exclusively the Lorenz gauge
$ \partial^{\mu} A_{\mu} = 0 \, , $
so that we have to assume that the data fulfill
$ \partial^{\mu} a_{\mu} = 0 $ . 
 
We want to prove local well-posedness of the Cauchy problem for (\ref{2.1}),(\ref{2.2}),(\ref{2.3}) for data in Fourier-Lebesgue spaces $\widehat{H}^{s,r}$ for $1<r \le 2$ with minimal regularity assumptions. These spaces are defined by its norm $\|u\|_{\widehat{H}^{s,r}} = \| \langle \xi \rangle^s \widehat{u}(\xi)\|_{L^{r'}}$ for dual exponents $r$ and $r'$  .

Chae-Chae \cite{CC} assumed $(\phi_0,\phi_1) \in H^2 \times H^1$ and proved local and even global well-posedness using energy conservation. This was improved by J. Yuan \cite{Y} to $(\phi_0,\phi_1),(a_{\mu 0},a_{\mu 1}),(N_0,N_1) \in H^s \times H^{s-1}$ with $s > \frac{3}{4}$ , who obtained a local solution $\phi,A_{\mu},N \in C^0([0,T],H^s) \cap C^1([0,T],H^{s-1})$, which is unique in a suitable subset of $X^{s,b}$-type. Using energy conservation this solution exists globally, if $s\ge 1$.

In \cite{P}  we further lowered  down the regularity of the data to $(\phi_0,\phi_1) \in H^s \times H^{s-1}$,  $(a_{\mu 0},a_{\mu 1}) \in H^{2s-\frac{3}{4}-} \times H^{2s-\frac{7}{4}-}$ , $(N_0,N_1) \in H^{\frac{1}{2}} \times H^{-\frac{1}{2}}$ on condition that $ s > \frac{1}{2} $ . We obtain a local solution $\phi \in C^0([0,T],H^s) \cap C^1([0,T],H^{s-1})$, $A_{\mu} \in C^0([0,T],H^{2s-\frac{3}{4}-}) \cap C^1([0,T],H^{2s-\frac{7}{4}-}) $ , $N\in C^0([0,T],H^{\frac{1}{2}}) \cap C^1([0,T],H^{-\frac{1}{2}-})$,  which is unique in a suitable subspace of $X^{s,b}$-type.

Whereas Chae-Chae only used standard energy type estimates Yuan applied bilinear Strichartz type estimates which were given in the paper of d'Ancona, Foschi and Selberg \cite{AFS}. We also use this type of estimates but additionally take advantage of a crucial null condition of the term $A_{\mu} \partial^{\mu} \phi$ in the wave equation for $\phi$. This was detected by Klainerman-Machedon \cite{KM} and Selberg-Tesfahun \cite{ST} for the Maxwell-Klein-Gordon equations and also by Selberg-Tesfahun \cite{ST1} for the corresponding problem for the Chern-Simons-Higgs equations.  In chapter 3 we prove bilinear estimates for the null forms and for general bilinear terms in generalized Bourgain-Klainerman-Machedon spaces $H^r_{s,b}$ (and $X^r_{s,b,\pm}$) based on estimates by Foschi and Klainerman \cite{FK}, Gr\"unrock \cite{G}, Grigoryan-Nahmod \cite{GN} and Grigoryan-Tanguay \cite{GT}.

The critical space for the data  with respect to scaling is $\widehat{H}^{\frac{2}{r}-1,r}$ for $\phi_0$ , $a_{\mu 0}$ and $N_0$ (cf. Remark 1 below). Thus in the case $r=2$ there is still a gap, which amounts to $\half$ for $\phi_0$ and $N_0$ and $\frac{1}{4}$ for $a_{\mu 0}$ .  
In order to shrink this gap we leave the $L^2$-based Sobolev spaces for the data and consider instead data in Fourier-Lebesgue spaces $\widehat{H}^{s,r}$ which coincide with the classical Sobolev spaces $H^s$ for $r=2$. We are especially interested in the case $r >1$ , but close to 1. In this case the minimal regularity assumptions for the data $(\phi(0),(\partial_t \phi)(0)) \in \widehat{H}^{s,r} \times \widehat{H}^{s-1,r}$, $(A_{\nu}(0),(\partial_t A_{\nu})(0)) \in \widehat{H}^{l,r} \times \widehat{H}^{l-1,r}$ and $(N(0),(\partial_t N)(0)) \in \widehat{H}^{m,r} \times \widehat{H}^{m-1,r}$ are $(s,l,m)=( \frac{21}{16}+,\frac{9}{8}+,\frac{9}{8}+)$ (cf. Theorem \ref{Theorem1}). Thus the gap for $(s,l,m)$ shrinks to $(\frac{5}{16},\frac{1}{8},\frac{1}{8})$ .

We now formulate our main results. One easily checks that a solution of (\ref{2.1}),(\ref{2.2}),(\ref{2.3}) (with (\ref{*}),(\ref{****})) under the Lorenz condition
\begin{equation}
\label{1.4*}
\partial^{\mu} A_{\mu} = 0
\end{equation}
also fulfills the following system
\begin{align}
\label{1.1}
&(\square +1)A_0 = -\kappa F_{12} -2e Im(\phi \overline{D_0 \phi}) +A_0 \\
\label{1.2}
&(\square +1)A_i = -\kappa \epsilon^{ij} F_{0j} -2e Im(\phi \overline{D_i \phi}) +A_i \\
\label{1.3}
&(\square +1)\phi =  2ie A_0 \partial_0 \phi -2ie A^j \partial_j \phi -e^2 A^j A_j \phi +e^2 A_0^2 \phi - U_{\overline{\phi}}(|\phi|^2,N) + \phi \\
\label{1.4}
&(\square +1)N = - U_N(|\phi|^2,N) + N \, .
\end{align}
Here we replaced $\square := \partial_t^2 - \Delta$ by $\square +1$ by adding a linear terms on both sides of the equations in oder to avoid the operator $(-\Delta)^{-\frac{1}{2}}$, which is unpleasent especially  in two dimensions.

Defining
\begin{align*}
&A_{\mu,\pm} = \frac{1}{2}(A_{\mu} \pm i^{-1} \langle \nabla \rangle^{-1} \partial_t A_{\mu}) &\hspace{-1em}\Leftrightarrow &A_{\mu}=A_{\mu,+} + A_{\mu,-} \,,\, \partial_t A_{\mu}=i\langle \nabla \rangle(A_{\mu,+}-A_{\mu,-}) \\
&\phi_{\pm} = \frac{1}{2}(\phi \pm i^{-1} \langle \nabla \rangle^{-1} \partial_t \phi) &\hspace{-1em}\Leftrightarrow &\phi = \phi_+ + \phi_- \, , \, \partial_t \phi = i \langle \nabla \rangle(\phi_+ - \phi_-) \\
&N_{\pm} = \frac{1}{2}(N \pm i^{-1} \langle \nabla \rangle^{-1} \partial_t N) &\hspace{-1em}\Leftrightarrow &N = N_+ + N_- \, , \, \partial_t N= i \langle \nabla \rangle(N_+ - N_-)
\end{align*}
we obtain the equivalent system
\begin{align}
\label{1.1'}
(i \partial_t \pm \langle \nabla \rangle) A_{0,\pm} & = \pm 2^{-1} \langle \nabla \rangle^{-1} ( \, {\mbox R.H.S.\, of} \,(\ref{1.1})) \\
\label{1.2'}
(i \partial_t \pm \langle \nabla \rangle) A_{j,\pm} & = \pm 2^{-1} \langle \nabla \rangle^{-1} ( \, {\mbox R.H.S.\, of} \,(\ref{1.2})) \\
\label{1.3'}
(i \partial_t \pm \langle \nabla \rangle) \phi_{\pm} & = \pm 2^{-1} \langle \nabla \rangle^{-1} ( \, {\mbox R.H.S.\, of} \,(\ref{1.3})) \\
\label{1.4'}
(i \partial_t \pm \langle \nabla \rangle) N_{\pm} & = \pm 2^{-1} \langle \nabla \rangle^{-1} ( \, {\mbox R.H.S.\, of} \,(\ref{1.4})) 
\end{align}

Denoting the Fourier transform $\mathcal{F}$ with respect to space as well as to space and time by $\,\widehat{}$ the operator
$\langle \nabla \rangle^{\alpha}$ is defined by $\mathcal{F}(\langle \nabla \rangle^{\alpha} f)(\xi) = \langle \xi \rangle^{\alpha} \widehat{f}(\xi)$, where $\langle \cdot \rangle := (1+|\cdot|^2)^{\frac{1}{2}}$ . Define
$a \pm := a \pm \epsilon$ for $\epsilon > 0$ sufficiently small.

We obtain the following result:                                  
\begin{theorem}
\label{Theorem1}
Let $1<r\le 2$ and $s,l,m \ge 1$ . Assume $ s > \frac{25}{16r}-\frac{1}{4}$ , $l > \frac{13}{8r}-\half$ , $m > \frac{13}{8r} - \half$, $s-1 \le l \le s+1$ , $s-1 \le m \le s+1$ and 
$2l-s > \frac{7}{4r}-1$ , $2s-l > \frac{3}{2r}$ , $2m-s > \frac{7}{4r} -1$ , $2s-m > \frac{7}{4r}-1$ . Assume
\begin{align*}
 &\phi_0 \in \widehat{H}^{s,r} \, , \, \phi_1 \in \widehat{H}^{s-1,r} \, , \, a_{\mu 0} \in \widehat{H}^{l,r} \, , \, a_{\mu 1} \in \widehat{H}^{l-1,r}  \, (\mu =0,1,2) \, ,\\ &n_0 \in \widehat{H}^{m,r} \, , \, n_1 \in \widehat{H}^{m-1,r}\, . 
\end{align*}
There exists  $T>0$ such that
the system (\ref{1.1}),(\ref{1.2}),(\ref{1.3}),(\ref{1.4}) with (\ref{*}),(\ref{****}) and Cauchy conditions
$$ \phi(0)= \phi_0 \, , \, \partial_t \phi(0) = \phi_1 \, , \,A_{\mu}(0) = a_{\mu 0} \, , \, \partial_t A_{\mu}(0) = a_{\mu 1} \, , \,\, N(0)=N_0 \, , \, \partial_t N(0) = N_1 $$
has a unique local solution 
\begin{align*}
& \phi \in X_+^{s,\frac{1}{r}+}[0,T] + X_-^{s,\frac{1}{r}+}[0,T] \\& A_{\mu} \in X_+^{l,\frac{1}{r}+}[0,T] +X_-^{l,\frac{1}{r}+}[0,T] \\ 
&N \in X^{m,\frac{1}{r}+}_+[0,T] + X_-^{m,\frac{1}{r}+}[0,T] \, .
\end{align*}
These spaces are defined in Definition \ref{Def.} below,
It has the properties
\begin{align*}
& \phi \in C^0([0,T],\widehat{H}^{s,r}) \cap C^1([0,T],\widehat{H}^{s,r}) \\
& A_{\mu} \in C^0([0,T],\widehat{H}^{l,r}) \cap C^1([0,T],\widehat{H}^{l-1,r}) \\
& N \in C^0([0,T],\widehat{H}^{m,r}) \cap C^1([0,T],\widehat{H}^{m-1,r}) \, .
\end{align*}
\end{theorem}
This result is proven in section 4. 

The following theorem was proven in \cite{P}.
\begin{theorem}
\label{Theorem1.1'}
Let $r=2$ . Assume that $1 >s>\half$ , $l=2s-\frac{3}{4}-$ , $m=\half$ . Then the statements of Theorem \ref{Theorem1} remain true.
\end{theorem}

By "interpolation" between the results of Theorem \ref{Theorem1} for $r=1+$ and Theorem \ref{Theorem1.1'} ($r=2$) we obtain the following improvement.

\begin{Cor}
\label{Cor.1}
Let $1<r\le 2$ . The statements of Theorem \ref{Theorem1} remain true in the case 
$s = \frac{13}{8r}-\frac{5}{16}+\epsilon$ , $l = \frac{7}{4r} - \frac{5}{8} +\epsilon$ and $m =\frac{5}{4r}-\frac{1}{8} + \epsilon$ , where $\epsilon > 0$ is sufficiently small.
\end{Cor} 

This is the minimal regularity for the data which is admissible by our method. Other combinations of $s$ , $l$ and $m$ are certainly possible, but we do not persue this.

In \cite{P} we obtained the following theorem as a consequence of 
Theorem \ref{Theorem1.1'}. The same proof applies for Theorem \ref{Theorem1} and Corollary \ref{Cor.1}.
 \begin{theorem}
\label{Theorem2}
Let the assumptions of Theorem \ref{Theorem1} or Theorem \ref{Theorem1.1'} be satisfied.
Moreover assume that the data fulfill
\begin{equation}
\label{GL}
\Delta a_{00} - \partial_1 a_{11} - \partial_2 a_{21} -\kappa(\partial_1 a_{20} - \partial_2 a_{10}) - 2e Im(\phi_0 \overline{\phi}_1) +2e^2 a_{00}|\phi_0|^2 = 0 
\end{equation} 
and
\begin{equation}
\label{21}
 \partial^{\mu} a_{\mu} = 0 \, . 
 \end{equation}
The solution of Theorem \ref{Theorem1} or Theorem \ref{Theorem1.1'} is the unique solution of the Cauchy problem
for the system
\begin{align}
\label{22}
& \partial_{\lambda} F^{\lambda \rho} + \frac{\kappa}{2} \epsilon^{\mu \nu \rho} F_{\mu \nu} +  2e Im( \phi  \overline{D^{\rho} \phi})  = 0\\
\label{23}
&D_{\mu} D^{\mu} \phi  + U_{\overline{\phi}}(|\phi|^2,N) = 0 \\
\label{24}
& \partial_{\mu} \partial^{\mu} N + U_N (|\phi|^2,N) = 0 \, ,
\end{align}
where
\begin{align}
\label{25}
U_{\overline{\phi}}(|\phi|^2,N) &= (e|\phi|^2+\kappa N)\phi + e^2(N+\frac{ev^2}{\kappa})^2 \phi \\
\label{25''}
 U_N(|\phi|^2,N) &= \kappa (e |\phi|^2 + \kappa N) +2e^2(N+\frac{ev^2}{\kappa}) |\phi|^2 
\end{align}
with initial conditions 
\begin{align}
\label{25'}
A_{\mu}(0) = a_{\mu 0} \, , \, \partial_tA_{\mu}(0) = a_{\mu 1} \, , \, \phi_0 = \phi_0 \, , \, \partial_t \phi(0) = \phi_1 \, , \, N(0)=N_0 \, , \, \partial_t N(0) = N_1 \, ,
\end{align}
 which fulfills the Lorenz condition $ \partial^{\mu} A_{\mu} = 0 \, . $
\end{theorem}

\begin{Def}
\label{Def.}
Let $1\le r\le 2$ , $s,b \in \R$ .  We recall, that the Fourier-Lebesgue space $\widehat{H}^{s,r}$ is defined by its norm $\|u\|_{\widehat{H}^{s,r}} = \| \langle \xi \rangle^s \widehat{u}(\xi)\|_{L^{r'}}$ for dual exponents $r$ and $r'$. The wave-Sobolev spaces $H^r_{s,b}$ are the completion of the Schwarz space ${\mathcal S}(\R^{1+3})$ with norm
$$ \|u\|_{H^r_{s,b}} = \| \langle \xi \rangle^s \langle  |\tau| - |\xi| \rangle^b \widehat{u}(\tau,\xi) \|_{L^{r'}_{\tau \xi}} \, , $$ 
where $r'$ is the dual exponent to $r$.
We also define $H^r_{s,b}[0,T]$ as the space of the restrictions of functions in $H^r_{s,b}$ to $[0,T] \times \mathbb{R}^3$.  Similarly we define $X^r_{s,b,\pm} $ with norm  $$ \|\phi\|_{X^r_{s,b\pm}} := \| \langle \xi \rangle^s \langle \tau \pm |\xi| \rangle^b \tilde{\phi}(\tau,\xi)\|_{L^{r'}_{\tau \xi}} $$ and $X^r_{s,b,\pm}[0,T] $ .\\
In the case $r=2$ we denote $H^2_{s,b} = : H^{s,b}$ and similarly $X^2_{s,b,\pm} = : X^{s,b}_{\pm}$ .  Remark that $\|u\|_{X^r_{s,b,\pm}} \le \|u\|_{H^r_{s,b}}$ for $b \le 0$ and the reverse estimate for $b \ge 0$.
\end{Def}

\noindent{\bf Remark 1:} The system (\ref{1.1})-(\ref{1.4}) is not scaling invariant, but ignoring lower order terms we can write it schematically as
\begin{align*}
\Box{A} & = \phi \nabla \phi + A \phi^2 \\
\Box{\phi} & = A \nabla \phi +A^2 \phi + N \phi^2 + \phi^3 \\
\Box{N} & = N \phi^2 \, .
\end{align*}
This system is invariant under the scaling 
$$ A_{\lambda}(x,t) = \lambda A(\lambda x,\lambda t) \, , \, \phi_{\lambda}(x,t) = \lambda \phi(\lambda x,\lambda t) \, , \,N_{\lambda}(x,t) = \lambda N(\lambda x,\lambda t) \, .$$ 
This implies
\begin{align*}
\|A_{\lambda}(0,\cdot)\|_{\dot{\widehat{H}}^{l,r}} &= \lambda^{1+l-\frac{2}{r}} \|a_{0}\|_{\dot{\widehat{H}}^{l,r}} \, , \\
\|\phi_{\lambda}(0,\cdot)\|_{\dot{\widehat{H}}^{s,r}} &= \lambda^{1+l-\frac{2}{r}} \|\phi_{0}\|_{\dot{\widehat{H}}^{s,r}}	\\
\| N_{\lambda}(0,\cdot)\|_{\dot{\widehat{H}}^{m,r}} &= \lambda^{1+m-\frac{2}{r}} \|N_0\|_{\dot{\widehat{H}}^{m,r}}               \, .
\end{align*}
Here $\|u\|_{\dot{\widehat{H}}^{s,r}} := \|| \xi|^s \widehat{u}(\xi)\|_{L^{r'}}$ , where $r$ and $r'$ are dual exponents.

Thus the critical data space with respect to scaling for  $A(0)$ , $\phi(0)$ , $N(0)$ (in dimension 2) is $\widehat{H}^{\frac{2}{r}-1,r}$. If we consider Theorem \ref{Theorem1} in the case $r=1+$ there remains a gap between this space and our minimal assumptions in Theorem \ref{Theorem1}, namely $A(0) \in H^{\frac{9}{8}+}$ , $\phi(0) \in \widehat{H}^{\frac{21}{16}+,r}$ , $N(0) \in H^{\frac{9}{8}+}$. Thus the gap amounts to $\frac{5}{16}$ for $s$ and $\frac{1}{8}$ for $l$ and $m$ .

However, we remark that this gap is considerably smaller compared to the gap for the case $r=2$ in Theorem \ref{Theorem1.1'} , which amounts to $\half$ for $s$ and $m$ and to $\frac{1}{4}$ for $l$  .

\section{Preliminaries}

We start by collecting some fundamental properties of the solution spaces. We rely on \cite{G}. The spaces $X^r_{s,b,\pm} $  are Banach spaces with ${\mathcal S}$ as a dense subspace. The dual space is $X^{r'}_{-s,-b,\pm}$ , where $\frac{1}{r} + \frac{1}{r'} = 1$. The complex interpolation space is given by
$$(X^{r_0}_{s_0,b_0,\pm} , X^{r_1}_{s_1,b_1,\pm})_{[\theta]} = X^r_{s,b,\pm} \, , $$
where $s=(1-\theta)s_0+\theta s_1$, $\frac{1}{r} = \frac{1-\theta}{r_0} + \frac{\theta}{r_1}$ , $b=(1-\theta)b_0 + \theta b_1$ . Similar properties has the space $H^r_{s,b}$ .\\
If $u=u_++u_-$, where $u_{\pm} \in X^r_{s,b,\pm} [0,T]$ , then $u \in C^0([0,T],\hat{H}^{s,r})$ , if $b > \frac{1}{r}$ .

The "transfer principle" in the following proposition, which is well-known in the case $r=2$, also holds for general $1<r<\infty$ (cf. \cite{GN}, Prop. A.2 or \cite{G}, Lemma 1). We denote $ \|u\|_{\hat{L}^p_t(\hat{L}^q_x)} := \|\tilde{u}\|_{L^{p'}_{\tau} (L^{q'}_{\xi})}$ .
\begin{prop}
\label{Prop.0.1}
Let $1 \le p,q \le \infty$ .
Assume that $T$ is a bilinear operator which fulfills
$$ \|T(e^{\pm_1 itD} f_1, e^{\pm_2itD} f_2)\|_{\hat{L}^p_t(\hat{L}^q_x)} \lesssim \|f_1\|_{\hat{H}^{s_1,r}} \|f_2\|_{\hat{H}^{s_2,r}}$$
for all combinations of signs $\pm_1,\pm_2$ , then for $b > \frac{1}{r}$ the following estimate holds:
$$ \|T(u_1,u_2)\|_{\hat{L}^p_t(\hat{L}^q_x)} \lesssim \|u_1\|_{H^r_{s_1,b}}  \|u_2\|_{H^r_{s_2,b}} \, . $$
\end{prop}

The following general local well-posedness theorem was given by  \cite{G}, Thm. 1.
\begin{theorem}
\label{Theorem0.3}
Let $N_{\pm}(u):=N_{\pm}(u_+,u_-)$ be a  multilinear functions.
Assume that for given $s \in \R$, $1 < r < \infty$ there exist $ b > \frac{1}{r}$ such that the estimate
$$ \|N_{\pm}(u)\|_{X^r_{s,b-1+,\pm}} \le \omega( \|u\|_{X^r_{s,b}}) $$
is valid with a nondecreasing function $\omega$ , where $\|u\|_{X^r_{s,b}} := \|u_-\|_{X^r_{s,b,-}} + \|u_+\|_{X^r_{s,b,+}}$. Then there exist $T=T(\|u_{0_ {\pm}}\|_{\hat{H}^{s,r}})$ $>0$ and a unique solution $(u_+,u_-)$ $ \in X^r_{s,b,+}[0,T] \times X^r_{s,b,-}[0,T] $ of the Cauchy problem
$$ \partial_t u_{\pm} \pm i\langle \nabla \rangle u = N_{\pm}(u)      \quad , \quad    u_{\pm}(0) = u_{0_{\pm}} \in \hat{H}^{s,r}       \, . $$
 This solution is persistent and the mapping data upon solution $(u_{0+},u_{0-})  \mapsto (u_+,u_-)$ , $\hat{H}^{s,r} \times \hat{H}^{s,r} \to X^r_{s,b,+}[0,T_0] \times X^r_{s,b,-}[0,T_0]  $ is locally Lipschitz continuous for any $T_0 < T$.
\end{theorem}

\section{Bilinear estimates}
The standard null forms are given by
$$
Q_{\alpha\beta}(u,v)=\partial_\alpha u \partial_\beta v-\partial_\beta u \partial_\alpha v.
$$
Let $q_{\alpha \beta}(u,v) := Q_{\alpha \beta}(|\nabla|^{-1}u,|\nabla|^{-1}v)$ , were $|\nabla|^{-1}$ has Fourier symbol $|\xi|^{-1}$ .
The proof of the following bilinear estimates relies on estimates given by Foschi and Klainerman \cite{FK}.
\begin{lemma}
\label{Lemma5.1}
Assume $0 \le\alpha_1,\alpha_2 $ ,  $\alpha_1+\alpha_2 \ge \frac{1}{r}$ and $ b > \frac{1}{r}$. The following estimate applies
$$ \|q_{12}(u,v)\|_{H^r_{0,0}} \lesssim \|u\|_{X^r_{\alpha_1,b,\pm_1}} \|v\|_{X^r_{\alpha_2,b,\pm_2}} \, . $$
\end{lemma}
\begin{proof}
	Because we use inhomogeneous norms it is obviously possible to assume $\alpha_1 + \alpha_2 = \frac{1}{r}$ . Moreover, by interpolation we may reduce to the case $\alpha_1= \frac{1}{r}$ , $\alpha_2 =0$ .
	
The left hand side of the claimed estimate equals
$$ \|{\mathcal F}(q_{12}(u,v))\|_{L^{r'}_{\tau \xi}} = \| \int q_{12}(\eta,\eta-\xi) \tilde{u}(\lambda,\eta) \tilde{v}(\tau - \lambda,\xi - \eta) d\lambda d\eta \|_{L^{r'}_{\tau \xi}} \, . $$
Let now $u(t,x) = e^{\pm_1 iD} u_0^{\pm_1}(x)$ , $v(t,x) = e^{\pm_2 iD} v_0^{\pm_2}(x)$ , so that 
$$ \tilde{u}(\tau,\xi) = c \delta(\tau \mp_1 |\xi|) \widehat{u_0^{\pm_1}}(\xi) \quad , \quad \tilde{v}(\tau,\xi) = c \delta(\tau \mp_2 |\xi|) \widehat{v_0^{\pm_2}}(\xi) \, . $$
This implies
\begin{align*}
&\|{\mathcal F}(q_{12}(u,v))\|_{L^{r'}_{\tau \xi}} \\
&= c^2 \| \int q_{12}(\eta,\eta-\xi) \widehat{u_0^{\pm_1}}(\eta) \widehat{v_0^{\pm_2}}(\xi-\eta) \,\delta(\lambda \mp_1 |\eta|) \delta(\tau-\lambda\mp_2|\xi-\eta|) d\lambda d\eta \|_{L^{r'}_{\tau \xi}} \\
& = c^2 \| \int q_{12}(\eta,\eta-\xi) \widehat{u_0^{\pm_1}}(\eta) \widehat{v_0^{\pm_2}}(\xi-\eta) \,\delta(\tau\mp_1|\eta| \mp_2|\xi-\eta|) d\eta \|_{L^{r'}_{\tau \xi}} \, .
\end{align*}
By symmetry we only have to consider the  elliptic case $\pm_1=\pm_2 = +$ and the hyperbolic case $\pm_1= + \, , \, \pm_2=-$ .  \\
{\bf Elliptic case.} We obtain by \cite{FK}, Lemma 13.2:
$$|q_{12}(\eta,\xi-\eta)| \le \frac{|\eta_1 (\xi - \eta)_2 - \eta_2 (\xi-\eta)_1|}{|\eta| \, |\xi - \eta|} \lesssim \frac{|\xi|^{\half} (|\eta| + |\xi - \eta| - |\xi|)^{\half}}{|\eta|^{\half} |\xi - \eta|^{\half}} \, . $$
By H\"older's inequality we obtain
\begin{align*}
&\|{\mathcal F}(q_{12}(u,v))\|_{L^{r'}_{\tau \xi}} \\
& \lesssim \|\int \frac{|\xi|^{\half} ||\tau|-|\xi||^{\half}}{|\eta|^{\half} |\xi - \eta|^{\half}} \,
 \delta(\tau-|\eta|-|\xi - \eta|) \, |\widehat{u_0^+}(\eta)| \, |\widehat{v_0^+}(\xi - \eta)| d\eta \|_{L^{r'}_{\tau \xi}} \\
& \lesssim \sup_{\tau,\xi} I  \,\, \|\widehat{D^{\frac{1}{r}} u_0^+}\|_{L^{r'}} \| \widehat{v_0^+}\|_{L^{r'}} \, ,
\end{align*}
where
$$ I = |\xi|^{\half} ||\tau|-|\xi||^{\half} \left( \int \delta(\tau - |\eta| - |\xi - \eta|) \, |\eta|^{-1-\frac{r}{2}} |\xi - \eta|^{-\frac{r}{2}} d\eta \right)^{\frac{1}{r}} \, . $$
We want to prove $ \sup_{\tau,\xi} I \lesssim 1 $ . By \cite{FK}, Lemma 4.3 we obtain
$$\int \delta(\tau - |\eta| - |\xi - \eta|) \, |\eta|^{-1-\frac{r}{2}} |\xi - \eta|^{-\frac{r}{2}} d\eta \sim \tau^A ||\tau|-|\xi||^B \, , $$
where $A= \max(1+\frac{r}{2},\frac{3}{2}) - 1-r= -\frac{r}{2}$ and $B=1-\max(1+\frac{r}{2},\frac{3}{2})=-\frac{r}{2}$ . Using $|\xi| \le |\tau|$  this implies
$$
I  \lesssim |\xi|^{\half} ||\tau|-|\xi||^{\half} \tau^{-\half} ||\tau|-|\xi||^{-\half}\le  1 \, .$$
{\bf Hyperbolic case.} We start with the following bound (cf. \cite{FK}, Lemma 13.2):
$$  |q_{12}(\eta,\xi-\eta)| \le \frac{|\eta_1 (\xi - \eta)_2 - \eta_2 (\xi-\eta)_1|}{|\eta| \, |\xi - \eta|}  \lesssim \frac{|\xi|^{\half} (|\xi|-||\eta|-|\eta-\xi||)^{\half}}{|\eta|^{\half} |\xi-\eta|^{\half}} \, , $$
so that similarly as in the elliptic case we have to estimate
$$ I = |\xi|^{\half} ||\tau|-|\xi||^{\half} \left( \int \delta(\tau - |\eta| + |\xi - \eta|) \, |\eta|^{-1-\frac{r}{2}} |\xi - \eta|^{-\frac{r}{2}} d\eta \right)^{\frac{1}{r}} \, . $$
In the subcase $|\eta|+|\xi-\eta| \le 2|\xi|$ we apply \cite{FK}, Prop. 4.5 and obtain
$$\int \delta(\tau - |\eta| + |\xi - \eta|) \, |\eta|^{-1-\frac{r}{2}} |\xi - \eta|^{-\frac{r}{2}} d\eta \sim |\xi|^A ||\xi|-|\tau||^B \, . $$
where in the subcase $0 \le \tau \le |\xi|$ we obtain $A=\max(\frac{r}{2},\frac{3}{2}) - 1-r = \half -r$ and $B= 1- \max(\frac{r}{2},\frac{3}{2})= -\frac{1}{2}$. \\
This implies
$$I \lesssim |\xi|^{\half} ||\tau|-|\xi||^{\half} |\xi|^{\frac{1}{2r}-1} ||\tau|-|\xi||^{-\frac{1}{2r}} \lesssim  1 \, . $$
Similarly in the subcase $-|\xi| \le \tau \le 0$ we obtain $A=\max(1+\frac{r}{2},\frac{3}{2})-1-r = - \frac{r}{2}$ , $B= 1 - \max(1+\frac{r}{2},\frac{3}{2}) = -\frac{r}{2} \, ,$ which implies
$$ I \sim |\xi|^{\half} ||\tau|-|\xi||^{\half} |\xi|^{-\half} ||\tau|-|\xi||^{-\half} = 1 \, .$$
In the subcase $|\eta| + |\xi-\eta| \ge 2|\xi|$ we obtain by \cite{FK}, Lemma 4.4:
 \begin{align*}
&\int \delta(\tau - |\eta| + |\xi - \eta|) \, |\eta|^{-1-\frac{r}{2}} |\xi - \eta|^{-\frac{r}{2}} d\eta \\
&\sim ||\tau|-|\xi||^{-\half} ||\tau|+|\xi||^{-\half}\int_2^{\infty} (|\xi|x+\tau)^{-\frac{r}{2}} (|\xi|x-\tau)^{1-\frac{r}{2}}(x^2-1)^{-\half} dx \\
&\sim  ||\tau|-|\xi||^{-\half} ||\tau|+|\xi||^{-\half} \int_2^{\infty} (x+\frac{\tau}{|\xi|})^{-\frac{r}{2}} (x-\frac{\tau}{|\xi|})^{1-\frac{r}{2}} (x^2-1)^{-\half} dx \, \cdot|\xi|^{1-r} \, .
\end{align*}
We remark that in fact the lower limit of the integral can be chosen as 2 by inspection of the proof in \cite{FK}.
The integral converges, because $|\tau| \le |\xi|$ and $r > 1.$  This implies the bound
$$ I \lesssim |\xi|^{\half} ||\tau|-|\xi||^{\half-\frac{1}{2r}}||\tau|+|\xi||^{-\frac{1}{2r}} |\xi|^{\frac{1}{r}-1} \lesssim  1 \, . $$
Summarizing we obtain
$$\|q_{12}(u,v)\|_{H^r_{0,0}} \lesssim \|D^{\frac{1}{r}} u_0^{\pm_1}\|_{L^{r'}}  \| v_0^{\pm_2}\|_{L^{r'}} \, . $$
By the transfer principle Prop. \ref{Prop.0.1} we obtain the claimed result. 
\end{proof}

In a similar manner we can also estimate the nullform $q_{0j}(u,v)$ .
\begin{lemma}
\label{Lemma5.2}
Assume $0 \le \alpha_1,\alpha_2 $ ,  $\alpha_1+\alpha_2 \ge \frac{1}{r}$ and $ b > \frac{1}{r}$ . The following estimate applies
$$ \|q_{0j}(u,v)\|_{H^r_{0,0}} \lesssim \|u\|_{X^r_{\alpha_1,b,\pm_1}} \|v\|_{X^r_{\alpha_2,b,\pm_2}} \, . $$
\end{lemma}
\begin{proof}
	Again we may  reduce to the case $\alpha_1= \frac{1}{r}$ and $\alpha_2=0$ .
Arguing as in the proof of Lemma \ref{Lemma5.1} we use in the elliptic case the estimate (cf. \cite{FK}, Lemma 13.2):
$$|q_{0j}(\eta,\xi-\eta)| \lesssim \frac{(|\eta|+|\xi-\eta|-|\xi|)^{\half}}{\min(|\eta|^{\half},|\xi-\eta|^{\half})} \, . $$
In the case  $|\eta| \le |\xi-\eta|$ we obtain
\begin{align*}
I &= ||\tau|-|\xi||^{\half} \left( \int \delta(\tau - |\eta| - |\xi - \eta|) \, |\eta|^{-1-\frac{r}{2}} d\eta \right)^{\frac{1}{r}} \\
&\sim ||\tau|-|\xi||^{\half} |\tau|^{\frac{A}{r}}  ||\tau|-|\xi||^{\frac{B}{r}} = 1\, ,
\end{align*}
because $ A=\max(1+\frac{r}{2},\frac{3}{2}) - 1-\frac{r}{2} = 0$ and $B= 1-\max(1+\frac{r}{2},\frac{3}{2})-\frac{r}{2} = -\frac{r}{2}$ . \\
In the case $|\eta| \ge |\xi-\eta|$ we obtain
\begin{align*}
I &= ||\tau|-|\xi||^{\half} \left( \int \delta(\tau - |\eta| - |\xi - \eta|) \, |\eta|^{-1} |\xi-\eta|^{-\frac{r}{2}} d\eta \right)^{\frac{1}{r}} \\
&\sim ||\tau|-|\xi||^{\half} |\tau|^{\frac{A}{r}}  ||\tau|-|\xi||^{\frac{B}{r}} (1+\log \frac{|\tau|}{||\tau|-|\xi||})^{\frac{1}{r}}\, ,
\end{align*}
where  $A= \max(1,\frac{r}{2},\frac{3}{2})-1-\frac{r}{2} = \half-\frac{r}{2}$ and $B=-\half$, so that
$$I \lesssim ||\tau|-|\xi||^{\half} \tau^{\frac{1}{2r}-\half} ||\tau|-|\xi||^{-\frac  {1}{2r}} \lesssim  1 \, .$$
 In the hyperbolic case we obtain by \cite{FK}, Lemma 13.2:
$$|q_{0j}(\eta, \xi-\eta)| \lesssim |\xi|^{\half} \frac{(|\xi|-||\eta|-|\eta-\xi||)^{\half}}{|\eta|^{\half} |\xi-\eta|^{\half}} $$
and argue exactly as in the proof of Lemma \ref{Lemma5.1}. The proof is completed as before.
\end{proof}

\begin{lemma}
\label{Lemma5.4}
Let $1 < r \le 2$ .
Assume $ \alpha_1,\alpha_2 \ge 0$ ,  $\alpha_1+\alpha_2 >\frac{3}{2r}$, $b_1 , b_2 > \frac{1}{2r}$, $b_1+b_2 > \frac{3}{2r}$. Then the following estimate applies:
$$\|uv\|_{H^r_{0,0}} \lesssim \|u\|_{H^r_{\alpha_1,b_1}} \|v\|_{H^r_{\alpha_2,b_2}} \, . $$
\end{lemma}
\begin{proof}
This follows from \cite{GT}, Prop. 3.1 by summation over the dyadic parts .
\end{proof}

\begin{lemma}
\label{Lemma5.5}
If $\alpha_1,\alpha_2,b_1,b_2 \ge 0$ , $\alpha_1+\alpha_2 > \frac{2}{r}$ and $b_1+b_2 > \frac{1}{r}$ the following estimate applies:
$$ \|uv\|_{H^r_{0,0}} \lesssim \|u\|_{H^r_{\alpha_1,b_1}} \|v\|_{H^r_{\alpha_2,b_2}} \, . $$
\end{lemma}
\begin{proof}
	We may assume $\alpha_1 = \frac{2}{r}+$ , $\alpha_2 =0$ , $b_1=\frac{1}{r}+$ , $b_2 =0$ (or similarly $b_1 =0$ , $b_2=\frac{1}{r}+$).
By Young's and H\"older's inequalities we obtain
\begin{align*}
\|uv\|_{H^r_{0,0}} &= \|\widehat{uv}\|_{L^{r'}_{\tau \xi}} \lesssim \|\widehat{u}\|_{L^{1}_{\tau \xi}}  \|\widehat{v}\|_{L^{r'}_{\tau\xi}}  \\
	& \lesssim \| \langle \xi \rangle^{-\frac{2}{r}-} \langle|\tau|-|\xi|\rangle^{-\frac{1}{r}-}\|_{L^r_{\tau\xi}}  \| \langle \xi \rangle^{\frac{2}{r}+} \langle|\tau|-|\xi|\rangle^{\frac{1}{r}+} \widehat{u}\|_{L^{r'}_{ \tau \xi}} \,  \|\widehat{v}\|_{L^{r'}_{\tau \xi}} \\
		& \lesssim \|u\|_{H^r_{\frac{2}{r}+,\frac{1}{r}+}} \|v\|_{H^r_{0,0}} \, .
		\end{align*}
\end{proof}

\begin{lemma}
\label{Lemma5.6}
Let $1 < r \le 2$ , $0 \le \alpha_1,\alpha_2$ and  $\alpha_1+\alpha_2 \ge \frac{1}{r}+b$ , $b >\frac{1}{r}$ . Then the following estimate applies:
$$ \|uv\|_{H^r_{0,b}} \lesssim
 \|u\|_{H^r_{\alpha_1,b}} \|v\|_{H^r_{\alpha_2,b}} \, . $$
\end{lemma}
\begin{proof}
	We may assume $\alpha_1 = \frac{1}{r}+b$ , $\alpha_2 =0$ .
We apply the "hyperbolic Leibniz rule" (cf. \cite{AFS}, p. 128):
\begin{equation}
\label{HLR}
 ||\tau|-|\xi|| \lesssim ||\rho|-|\eta|| + ||\tau - \rho| - |\xi-\eta|| + b_{\pm}(\xi,\eta) \, , 
\end{equation}
where
 $$ b_+(\xi,\eta) = |\eta| + |\xi-\eta| - |\xi| \quad , \quad b_-(\xi,\eta) = |\xi| - ||\eta|-|\xi-\eta|| \, . $$

Let us first consider the term $b_{\pm}(\xi,\eta)$ in (\ref{HLR}). Decomposing as before $uv=u_+v_++u_+v_-+u_-v_++u_-v_-$ , where $u_{\pm}(t)= e^{\pm itD} f , v_\pm(t) = e^{\pm itD} g$  , we use
$$ \widehat{u}_{\pm}(\tau,\xi) = c \delta(\tau \mp |\xi|) \widehat{f}(\xi) \quad , \quad  \widehat{v}_{\pm}(\tau,\xi) = c \delta(\tau \mp |\xi|) \widehat{g}(\xi)$$
and have to estimate
\begin{align*}
&\| \int  b^{b}_{\pm}(\xi,\eta) \delta(\tau - |\eta| \mp |\xi-\eta|) \widehat{f}(\xi) \widehat{g}(\xi-\eta)d\eta \|_{L^{r'}_{\tau \xi}} \\
& = \| \int ||\tau|-|\xi||^b \delta(\tau - |\eta| \mp |\xi-\eta|) \widehat{f}(\xi) \widehat{g}(\xi-\eta) d\eta \|_{L^{r'}_{\tau \xi}} \\
& \lesssim \sup_{\tau,\xi} I \, \|\widehat{D^{\frac{1}{r}+b} f}\|_{[L^{r'}}
\|\widehat{g}\|_{[L^{r'}} \, .
\end{align*}
Here we used H\"older's inequality, where
$$I = ||\tau|-|\xi||^b  (\int \delta(\tau-|\eta|\mp|\xi-\eta|) |\eta|^{-1-b r} d\eta)^{\frac{1}{r}} \, . $$
In order to obtain $I \lesssim 1$ we first consider the elliptic case $\pm_1=\pm_2=+$ and use \cite{FK}, Prop. 4.3. Thus 
$$ I \sim ||\tau|-|\xi||^b  \tau^{\frac{A}{r}} ||\tau|-|\xi||^{\frac{B}{r}} = ||\tau|-|\xi||^b ||\tau|-|\xi||^{-b} = 1$$
with $A=\max(1+br,\frac{3}{2})-(1+br)=0$ and $B=1-\max(1+br,\frac{3}{2})= -br$ .

Next we consider the hyperbolic case $\pm_1 = + \, , \, \pm_2=-$ . \\
First we assume $|\eta|+|\xi-\eta| \le 2 |\xi|$ and use \cite{FK}, Prop. 4.5 which gives
$$\int \delta(\tau - |\eta| + |\xi-\eta|) |\eta|^{-1-br}  d\eta \sim |\xi|^A ||\xi|-|\tau||^B \, , $$
where $A=\frac{3}{2}-(1+br)= \half-br$ , $B=1-\frac{3}{2}=- \half$ , if $0 \le \tau\le |\xi|$ ,so that
$$I \sim ||\tau|-|\xi||^b |\xi|^{\frac{1}{2r}-b} ||\tau|-|\xi||^{-\frac{1}{2r}} \lesssim 1 \, .$$   
If $-|\xi| \le \tau \le 0$ we obtain $A=\max(1+br,\frac{3}{2})-(1+br)=0$ , $B=1-\max(1+br,2)=-br$ , which implies  $I \lesssim 1$ .\\
Next we assume $|\eta|+|\xi-\eta| \ge 2|\xi|$ , use \cite{FK}, Lemma 4.4 and obtain
\begin{align*}
&I \sim ||\tau|-|\xi||^b (\int \delta(\tau-|\eta|-|\xi-\eta|) |\eta|^{-1-br} d\eta)^{\frac{1}{r}} \\
& \sim  ||\tau|-|\xi||^b (||\tau|-|\xi||^{-\half} ||\tau|+|\xi||^{-\half} \int_2^{\infty} (|\xi|x + \tau)^{-br} (|\xi|x-\tau)(x^2-1)^{-\half} dx)^{\frac{1}{r}} \\
& \sim ||\tau|-|\xi||^b \, \cdot\\
&\quad \cdot (||\tau|-|\xi||^{-\half} ||\tau|+|\xi||^{-\half}\int_2^{\infty} (x+\frac{\tau}{|\xi|})^{-br} (x-\frac{\tau}{|\xi|}) (x^2-1)^{-\half} dx \,\cdot |\xi|^{1-br})^{\frac{1}{r}}  \, .
\end{align*}
This integral converges, because $\tau \le |\xi|$ and $b >\frac{1}{r}$ .This implies
$$I \lesssim  ||\tau|-|\xi||^{b-\frac{1}{2r}} ||\tau|+|\xi||^{-\frac{1}{2r}} |\xi|^{\frac{1}{r}-b} \lesssim  1 \, , $$
using $|\tau| \le |\xi|$ .

By the transfer principle we obtain
$$\| B_{\pm}^b (u,v)\|_{X^r_{0,0}} \lesssim \|u\|_{X^r_{\frac{1}{r}+b,b,\pm_1}} \|v\|_{X^r_{0,b,\pm_2}} \, .$$ 
 Here $B^b_{\pm}$ denotes the operator with Fourier symbol $b_{\pm}$ . \\
Consider now the term $||\rho|-|\eta||$ (or similarly $||\tau-\rho|-|\xi-\eta||$) in (\ref{HLR}). We have to prove
$$ \|u D_-^b v\|_{H^r_{0,0}} \lesssim \|u\|_{X^r_{\alpha_1,b,\pm_1}} \|v\|_{X^r_{\alpha_2,b,\pm_2}} \, , $$
which is implied by
$$\|uv\|_{H^r_{0,0}} \lesssim \|u\|_{X^r_{\alpha_1,b,\pm_1}} \|v\|_{X^r_{\alpha_2,0,\pm_2}} \, . $$
This results from Lemma \ref{Lemma5.5}, because $\alpha_1+\alpha_2\ge \frac{1}{r} +b > \frac{2}{r}$ , which completes the proof. 
\end{proof}

As a consequence we obtain the following estimate for cubic nonlinearities.

\begin{lemma}
\label{cubic}
Let $1<r\le 2$ , $ b >\frac{1}{r}$ , $s_0\ge 1$ , $s_0 \le s_1+1$ , $ 2 s_1-s_0 > \frac{7}{4r}-1$ and $s_1 > \frac{13}{8r} - \half$ . Then the following estimate applies:
$$\|uvw\|_{H^r_{s_0-1,0}} \lesssim \|u\|_{H^r_{s_1,b}} \|v\|_{H^r_{s_1,b}} \|w\|_{H^r_{s_0,b}} \, . $$
\end{lemma}
\begin{proof}
By Lemma \ref{Lemma5.6} we obtain
$$\|uv\|_{H^r_{0,\frac{1}{r}+}} \lesssim \|u\|_{H^r_{\alpha_1,b}} \|v\|_{H^r_{\alpha_2,b}} \, ,$$
if $\alpha_1+\alpha_2 >{2}{r}$ and $r > \frac{1}{r}$ , and by Lemma \ref{Lemma5.4}
$$\|uv\|_{H^r_{0,0}} \lesssim \|u\|_{H^r_{\alpha_1,b_1}} \|v\|_{H^r_{\alpha_2,b_2}} \, ,$$
if $\alpha_1+\alpha_2 >\frac{3}{2r}$ , $b_1,b_2 > \frac{1}{2r}$ and $b_1+b_2 >\frac{3}{2r}$ .
By interpolation this implies
$$\|uv\|_{H^r_{0,\frac{1}{2r}+}} \lesssim \|u\|_{H^r_{\alpha_1,b}} \|v\|_{H^r_{\alpha_2,b}} \, ,$$
if $\alpha_1+\alpha_2 > \frac{7}{4r}$ and $b >\frac{1}{r}$ .
Using first Lemma \ref{Lemma5.4} again und then the last inequality we obtain
\begin{align*}
\|uvw\|_{H^r_{s_0-1,0}} \lesssim \|uv\|_{H^r_{k,\frac{1}{2r}+}} \|w\|_{H^r_{s_0,b}} \lesssim \|u\|_{H^r_{s_1,b}} \|v\|_{H^r_{s_1,b}} \|w\|_{H^r_{s_0,b}} \, ,
\end{align*}
where  $k+1> \frac{3}{2r}$ and $2s_1-k > \frac{7}{4r}$ , which requires $s_1 > \frac{13}{8r} -\half$ , and $k \ge s_0-1$, thus $ 2 s_1-s_0 > \frac{7}{4r}-1$ .

\end{proof}

\section{Proof of Theorem \ref{Theorem1}}
\begin{proof}[Proof of Theorem \ref{Theorem1}]
An application of the contraction mapping is by an obvious generalization of Theorem \ref{Theorem0.3} to systems reduced to suitable multilinear estimates of the right hand sides of (\ref{2.1}),(\ref{2.2}) and (\ref{2.3}).

The linear terms are easily treated and therefore omitted here.

We now consider the right hand side of (\ref{1.1'}):
$$ -2e Im(\phi \overline{D_0 \phi}) = -2e Im(\phi_++\phi_-)(-i)\langle \nabla \rangle(\overline{\phi}_+ - \overline{\phi}_-)) -2e^2 A_0 |\phi|^2 \, . $$
By Lemma \ref{Lemma5.4} we obtain
$$ \|\phi_{\pm 1} \langle \nabla \rangle \overline{\phi}_{\pm 2} \|_{H^r_{l-1,0}} \lesssim \|\phi_{\pm 1} \|_{X^r_{s,b,\pm 1}}  \|\langle \nabla \rangle \phi_{\pm 2} \|_{X^r_{s-1,b,\pm 2}} \, , $$
if $2s-l > \frac{3}{2r}$ and $s \ge l-1$ ,
where here and in the following $\pm_1$ and $\pm_2$ denote independent signs.
\\
The cubic term  is handled as follows. 
By Lemma \ref{cubic} we obtain
$$\|A_0 |\phi|^2\|_{H^r_{l-1,0}} \lesssim \|A_0\|_{H^r_{l,b}} \|\phi\|_{H^r_{s,b}}^2 \,, $$
because $s \ge l-1$ , $ s > \frac{13}{8r}-\half$ and $2s-l > \frac{7}{4r} - 1$ .

The right hand side of (\ref{1.2'}) can be handled in the same way.

It remains to consider the right hand side of (\ref{1.3'}). We start with the most interesting quadratic term, where a null condition comes into play, namely
$ 2ie A_{\mu} \partial^{\mu} \phi  \, . $
Defining the modified Riesz transforms $R_j := \langle \nabla \rangle^{-1} \partial_j$ and splitting $A_j$ into divergence-free and curl-free parts and a smooth remainder we obtain
$$A_j = A_j^{df} + A_j^{cf} + \langle \nabla \rangle^{-2} A_j \, , $$
where
\begin{align*}
& A_1^{df} = R_2(R_1 A_2-R_2 A_1) & & A_2^{df} = R_1(R_2A_1-R_1 A_2) \\
& A_1^{cf} = -R_1(R_1A_1 + R_2 A_2) & &A_2^{cf} = -R_2(R_1A_1+R_2A_2) \, . 
\end{align*}
Now we have with $\mathbf A = (A_1,A_2)$  (these arguments go back to Selberg-Tesfahun) :
\begin{align*}
	A^\mu \partial_\mu \phi
	&=
	\left( - A_0 \partial_t \phi
	+ \mathbf A^{\text{cf}} \cdot \nabla \phi \right)
	+ \mathbf A^{\text{df}} \cdot \nabla \phi +  \angles{\nabla}^{-2} \mathbf A  \cdot \nabla \phi
	\end{align*}
	Let us first consider the first term in the parentheses. 
	We use the Lorenz gauge, $\partial_t A_0=\nabla \cdot \mathbf A  $,  to write
	\begin{align*}
	\mathbf A^{\text{cf}} \cdot \nabla \phi&
	=-\angles{\nabla}^{-2} \partial^i(\partial_t A_0) \partial_i \phi=
	- \partial_t ( \angles{\nabla}^{-1} R^i A_0)\partial_i \phi.
	\end{align*}
	We can also write
	\begin{align*}
	A_0 \partial_t \phi 
	&=-\angles{\nabla}^{-2}  \partial_i\partial^i A_0 \partial_t \phi+ \angles{\nabla}^{-2}  A_0 \partial_t \phi
	\\
	&=-\partial_i(\angles{\nabla}^{-1} R^i A_0) \partial_t \phi
	+  \angles{\nabla}^{-2}  A_0 \partial_t \phi.
	\end{align*}
	Combining the above identities, we get
	\begin{align*}
	-A_0 \partial_t \phi +  \mathbf A^{\text{cf}} \cdot \nabla \phi
	&=Q_{i0}(\angles{\nabla}^{-1} R^i A_0, \phi)- \angles{\nabla}^{-2}  A_0 \partial_t \phi.
	\end{align*}
	
	Next, we consider the second term. 
	We have 
	\begin{align*}
	\mathbf A^{\text{df}} \cdot \nabla \phi
	&= \angles{\nabla}^{-2} \partial_2 (\partial_1A_2-\partial_2 A_1) \partial_1 \phi - \angles{\nabla}^{-2} \partial_1(\partial_1 A_2 - \partial_2 A_1) \partial_2 \phi
	\\
	&=-  Q_{12}(\angles{\nabla}^{-2}(\partial_1 A_2-\partial_2 A_1),\phi)
	\\
	&=
	-  Q_{12}\left(\angles{\nabla}^{-1}( R^1 A_2 -R^2 A_1), \phi\right).
	\end{align*}
	
	Thus, we have shown 
	\begin{equation}\label{Null2}
	\begin{split}
	A^\mu \partial_\mu \phi
	&= -  Q_{12}\left(\angles{\nabla}^{-1}( R^1 A_2 -R^2 A_1), \phi\right)
	 \\
	& \qquad + Q_{i0}(\angles{\nabla}^{-1} R^i A_0, \phi)+\angles{\nabla}^{-2}  A^\alpha \partial_\alpha \phi.
	\end{split}
	\end{equation}

Our aim is to prove the following estimate :
\begin{equation}
\label{A}
\|A_{\mu,\pm 1} \partial^{\mu} \phi_{\pm 2}\|_{H^r_{s-1,0}} \lesssim \sum_{\mu} \|A_{\mu,\pm 1} \|_{X^r_{l,b,\pm 1}} \| \langle \nabla \rangle \phi_{\pm 1} \|_{X^r_{s-1,b,\pm 2}} \, .
\end{equation}

We first estimate the last term on the right hand side of (\ref{Null2}). By Lemma \ref{Lemma5.5} we obtain the sufficient estimate
$$ \| \langle \nabla \rangle^{-2} A_{j,\pm 1} \partial^j \phi_{\pm 2}\|_{H^r_{s-1,0}} \lesssim \|A_{j,\pm 1} \|_{X^r_{l,b,\pm 1}} \|\partial^j \phi_{\pm 2}\|_{H^r_{s-1,b,\pm 2}} \, , $$
if $ l > \frac{2}{r}-2$ .

The estimate for the first two terms on the right hand side of (\ref{Null2})
reduces to
$$\|Q^{12}(A_{\mu}, \langle \nabla \rangle \phi)\|_{H^r_{s-1,0}} + \|Q^{i0}(A_{\mu}, \langle \nabla \rangle \phi)\|_{H^r_{s-1,0}} \lesssim \|A_{\mu}\|_{X^r_{l,b,\pm_1}} \| \langle \nabla \rangle \phi\|_{X^r_{s-1,b,\pm_2}} \, , $$
which follows from Lemma \ref{Lemma5.1}, if $l \ge s-1$ and $ l \ge \frac{1}{r}$ .

Concerning $A^2_{\mu} \phi$ we obtain by Lemma \ref{cubic} :
$$ \|A^2_{\mu} \phi\|_{H^r_{s-1,0}} \lesssim \|A_{\mu}\|_{H^r_{l,b}}^2 \|\phi\|_{H^r_{s,b}} $$
provided $l > \frac{13}{8r}-\half$ , $s \le l+1$ and $2l-s > \frac{7}{4r}-1$ .

For the terms $|\phi|^2 \phi$ and $N^2 \phi$ we obtain similarly
$$ \|\phi|^2 \phi\|_{H^r_{s-1,0}} \lesssim \|\phi\|_{X^r_{s,b}}^3 \, ,$$
if $s >  \frac{13}{8r}-\half$ , and
$$\|N^2 \phi\|_{H^r_{s-1,0}} \lesssim \|N\|_{H^r_{m,b}}^2 \|\phi\|_{H^r_{s,b}} \, , $$
if  $m > \frac{13}{8r}-\half$ , $ s \le m+1$ and $2m-s > \frac{7}{4r} -1$ .

The term $N\phi$  can be handled even more easily.

Finally we have to consider the terms on the right hand side of (\ref{1.4'}). The term  $N|\phi|^2$ (and similarly $|\phi|^2$) is treated by Lemma \ref{cubic} as follows:
$$\|N|\phi|^2\|_{H^r_{m-1,0}} \lesssim \|N\|_{H^r_{m,b}} \|\phi|_{H^r_{s,b}}^2 \, ,$$
if $s >  \frac{13}{8r}-\half$ , $m \le s+1$ and $2s-m > \frac{7}{4r}-1$ . 
The proof of Theorem \ref{Theorem1} is now complete.
\end{proof}

\begin{proof}[Proof of Corollary \ref{Cor.1}]
Corollary \ref{Cor.1} follows by an easy calculation using the interpolation properties of the spaces $X^r_{s,b,\pm}$ by interpolation between the bi- and trilinear estimates for the nonlinearities, which are given above for $r=1+$ and in \cite{P} for $r=2$ .
\end{proof}

\end{document}